\documentclass[12pt]{amsart}
\usepackage{amssymb,amsmath,euscript,enumerate,tikz}
\usepackage[margin=1in]{geometry} 
\usepackage{caption}
\usepackage{hyperref}
\usepackage{float}
\usepackage{xcolor}
\usepackage{tikz-cd}
\tikzstyle{vertex}=[circle,draw, inner sep=0pt, minimum size=1.5pt]

\newtheorem{theorem}{Theorem}[section]
\newtheorem{lemma}[theorem]{Lemma}

\newtheorem{question}[theorem]{Question}
\newtheorem{conjecture}[theorem]{Conjecture}
\newtheorem{corollary}[theorem]{Corollary}

\theoremstyle{definition}
\newtheorem{definition}[theorem]{Definition}

\newtheorem{example}[theorem]{Example}

\newcommand\Hilb{\operatorname{Hilb}}
\newcommand\cdeg{\operatorname{cdeg}}

\newtheorem{remark}[theorem]{Remark}
\usepackage{color,soul}

\newcommand{\pd}{\operatorname{pd}}
\newcommand{\iv}{\operatorname{iv}}
\newcommand{\qc}{\operatorname{qc}}

\newcommand{\reg}{\mathrm{reg}}

\begin{document}
	
	\title[Binomial edge ideals and bounds for their regularity]{Binomial edge ideals and bounds for their regularity}
	
	\author[Arvind Kumar]{Arvind Kumar}
	\email{arvkumar11@gmail.com}
	\address{Department of Mathematics, Indian Institute of Technology Madras, Chennai, 600036, India}

	\begin{abstract}
	Let $G$ be a simple graph on $n$ vertices and $J_G$ denote the
	corresponding binomial edge ideal in $S = K[x_1, \ldots, x_n, y_1,
	\ldots, y_n].$ We prove  that the Castelnuovo-Mumford regularity of $J_G$
	is bounded above by $c(G)+1$, when $G$ is  a  quasi-block graph or semi-block graph. We give another proof of  Saeedi Madani-Kiani regularity upper bound
	conjecture for chordal graphs. We obtain the regularity of binomial edge ideals of  Jahangir graphs. Later, we establish a sufficient condition for Hibi-Matsuda conjecture to be true.

	\end{abstract}
	\maketitle
\section{Introduction}
 Let $G$ be a simple graph on $[n]$ and
 $S=K[x_1, \ldots, x_n,y_1, \ldots, y_n]$, where $K$ is a field.  The
 binomial edge ideal of the graph $G$, $J_G =(x_i y_j - x_j y_i :
 \{i,j\} \in E(G), \; i <j)$, was  introduced by Herzog et al. in \cite{HH1}
and independently by Ohtani in \cite{oh}. Since then researchers have been 
trying to study the algebraic invariants of $J_G$ in terms of  the combinatorial
invariants of $G$. In  \cite{her1,HH1,JNR,KM3,MM,Rauf,KM1,KM2}, the authors
have  established connections between homological invariants such as depth,
codimension, Betti numbers and Castelnuovo-Mumford regularity of $J_G$
with certain combinatorial invariants associated with the graph $G$. The study of
Castelnuovo-Mumford regularity of binomial edge ideals has attracted a lot of 
attention in the recent past due to its algebraic and geometric importance.  In \cite[Theorem
1.1]{MM}, Matsuda and Murai proved that for any graph $G$ on $[n]$, $l(G) \leq \reg(S/J_G) \leq n-1$, where $l(G)$ is the length of a
longest induced path in $G$. In the same article, they conjectured that $\reg(S/J_G) = n-1$
if and only if $G=P_n$. This conjecture was settled in
affirmative by Kiani and Saeedi Madani in \cite{KM3}. For a graph $G$,
let $c(G)$ denote the number of maximal cliques in $G$. If $G$ is a
closed graph, i.e., the generators of $J_G$ are a Gr\"obner basis  with respect to lexicographic order induced by $x_1 > \cdots >x_n>y_1>\cdots >y_n$, then
Saeedi Madani and Kiani \cite{KM1} proved that $\reg(S/J_G) \leq c(G)$.  In \cite{KM2}, the following conjecture was proposed.
\begin{conjecture}\label{con}
	Let $G$ be a graph on $[n]$. Then, $\reg(S/J_G) \leq c(G)$.
\end{conjecture}
In \cite{EZ},
Ene and Zarojanu proved the conjecture for block graphs. In \cite{JACM}, 
Jayanthan and Kumar proved the conjecture for  $k$-fan graph of the  complete graph. In \cite{MKM2018}, Rouzbahani Malayeri et al. proved the conjecture for the class of
chordal graphs. Recently, in \cite{KK19}, Kahle and Kr\"usemann proved the conjecture for cographs.   In the third section, we prove Saeedi Madani-Kiani conjecture for some classes of non-chordal graphs. We prove Conjecture \ref{con} for the class of quasi-block graphs (see Sect. $3$ for the definition).  Indeed, we give an example of a quasi-block graph to show that the  upper bound is tight. Then, we prove Conjecture \ref{con} for semi-block graphs (see Sect. $3$ for the definition).  We then give another proof of Conjecture \ref{con} for the class of chordal graphs. Also, we provide
a sufficient condition for chordal graphs so that the inequality is strict in Conjecture \ref{con}.

In the fourth section,  we obtain the regularity of binomial edge ideals of  Jahangir graphs (see Sect. $4$ for the definition). To compute the regularity of binomial edge ideals of Jahangir graphs, we use \cite[Theorem 2.1]{ERT19}, which is a recent result due to Ene, Rinaldo and Terai.

In \cite{HibiM2018}, Hibi and Matsuda studied the regularity of binomial edge ideals of graphs from  the
algebraic perspective and conjectured  that the regularity is bounded above by the degree of the $h$-polynomial of $S/J_G$.
\begin{conjecture}\cite[Conjecture 0.1]{HibiM2018}\label{Hibi-con}
Let $G$ be a graph on $[n]$. Then, $\reg(S/J_G)\leq \deg h_{S/J_G}(t)$.
\end{conjecture}
In the last section, we obtain a sufficient condition for Conjecture
\ref{Hibi-con} to be true. Recently, in \cite{KK19}, Kahle and
Kr\"usemann found a counterexample to Conjecture \ref{Hibi-con}. We
give another example to show that
Conjecture \ref{Hibi-con} is not true in general, even for chordal graphs. 
\section{Preliminaries}
In this section, we  recall  some notation and fundamental results on
graphs and the corresponding binomial edge ideals which are used
throughout this paper.

Let $G$  be a  finite simple graph with vertex set $V(G)$ and edge set
$E(G)$. For $A \subseteq V(G)$, $G[A]$ denotes the \textit{induced
subgraph} of $G$ on the vertex set $A$, i.e., for $i, j \in A$, $\{i,j\} \in E(G[A])$ if and only if $ \{i,j\} \in E(G)$. 
For a vertex $v$, $G \setminus v$ denotes the  induced subgraph of $G$
on the vertex set $V(G) \setminus \{v\}$. A vertex $v \in V(G)$ is
said to be a \textit{cut vertex} if $G \setminus v$ has  more
connected components than $G$.
A subset $U$ of $V(G)$ is said to be a 
\textit{clique} if $G[U]$ is a complete graph. For $v \in V(G)$, let $\cdeg_G(v)$ denote the number of maximal cliques which contains $v$.   We say that $G$ is \textit{k-vertex connected}, if $k <|V(G)|$ and for every subset $S \subset V(G)$ such that $|S| < k$, the induced subgraph $G[V(G)\setminus{S}]$ is connected. The \textit{vertex connectivity} of $G$, denoted by $\kappa(G)$, is defined as the maximum integer $k$ such that G is $k$-vertex connected.

A \textit{simplicial complex} $\Delta$ on the vertex set $[n]$ is a collection of subsets of $[n]$ such that:
$(i)$ $\{v\} \in \Delta$  for all $v \in [n]$; $(ii)$ $F \in \Delta$ and $G\subseteq F$ implies $G \in \Delta$.
An element $F \in \Delta$ is called a \textit{face} of $\Delta$. A maximal face of $\Delta$  with respect to inclusion is called a \textit{facet} of $\Delta$.
A facet $F$ of $\Delta$ is called a {\em leaf}, if either $F$ is the only facet, or else there exists a 
facet $G$ such that for each facet $H$ of $\Delta$ with $H\neq F$, $H\cap F\subsetneq G\cap F$. A vertex $v$ is said to be a \textit{free vertex  (simplicial vertex)} if it belongs to
exactly one facet of $\Delta(G)$. Each leaf $F$ has at least one  free vertex. 

The simplicial complex $\Delta$ is  called a {\em quasi-forest}, if its facets can be ordered $F_1,\ldots,F_s$
such that for all $i>1$, the facet $F_i$ is a leaf of the simplicial complex with facets $F_1,\ldots, F_{i-1}$.
Such an order of the facets is called a {\em leaf order}. A connected quasi-forest is called a {\em quasi-tree}.
The collection of all cliques of a  graph $G$ form a simplicial complex which is called  {\em clique complex } of $G$ and is denoted by $\Delta(G).$  Its facets are the maximal cliques of $G.$ The \textit{clique number}
of a graph ${G}$, denoted by $\omega({G})$, is the maximum size of the maximal cliques of ${G}$.

  A vertex $v$ is said to be an \textit{internal vertex}, if $v$ is not a free vertex. Let $\iv(G)$ denote the number of 
  internal vertices of $G$. The neighborhood of a vertex $v$, denoted by $N_G(v),$ is defined as $\{u \in V(G) 
  :  \{u,v\} \in E(G)\}$.  For a vertex $v$,  $G_v$ is the
graph on the vertex set $V(G)$ and edge set $E(G_v) =E(G) \cup \{
\{u,w\}: u,w \in N_G(v)\}$. For $e \in E(G)$, $G\setminus e$ is 
the graph on the vertex set $V(G)$ and edge set $E(G) \setminus \{e\}$. 

For  $T \subset [n]$, let $\bar{T} = [n]\setminus T$ and $c_G(T)$
denote the number of connected components of $G[\bar{T}]$. Let $G_1,\ldots,G_{c_{G}(T)}$ be the connected 
components of $G[\bar{T}]$. For each $i$, let $\tilde{G_i}$ denote the complete graph on $V(G_i)$. Set
$P_T(G) = (\underset{i\in T} \cup \{x_i,y_i\}, J_{\tilde{G_1}},\ldots, J_{\tilde{G}_{c_G(T)}}).$
In \cite{HH1}, it was shown by Herzog et al.  that $J_G =  \underset{T \subseteq [n]}\cap P_T(G).$

The following basic property of   regularity is used repeatedly in this
article. We refer the reader to the book \cite[Chapter 18]{Peeva} for more properties on regularity.

\begin{lemma}\label{regularity-lemma}
	Let ${M},{N}$ and ${P}$ be finitely generated graded ${S}$-modules. 
	If $$ 0 \rightarrow {M} \xrightarrow{f}  {N} \xrightarrow{g} {P} \rightarrow 0$$ is a 
	short exact sequence with $f,g$  
	graded homomorphisms of degree zero, then 
	\begin{enumerate}
		\item $\reg({M}) \leq \max\{\reg({N}),\reg({P})+1\}$.
		\item $\reg({M}) = \reg({N})$, if $\reg({N}) >\reg({P})$. 
	\end{enumerate}
\end{lemma}
\section{Saeedi Madani-Kiani conjecture}
In this section, we prove Conjecture \ref{con} for some classes of non-chordal graphs. Also, we give another proof of Conjecture \ref{con} for the class of chordal graphs. We begin by recalling a lemma by Ohtani which is 
highly useful in computing the regularity of binomial edge ideals.

\begin{lemma}\cite[Lemma 4.8]{oh}\label{ohlemma}
	Let $G$ be a  graph. If $v$ is an internal vertex, then $$J_G =J_{G_v} \cap ((x_v,y_v)+J_{G\setminus v}).$$
\end{lemma}
Note that $J_{G_v} +((x_v,y_v)+J_{G\setminus v}) = (x_v,y_v)+J_{G_v \setminus v}$.
Therefore, we have the following short exact sequence:
 \begin{equation}\label{ses1}
0  \longrightarrow  \dfrac{S}{J_{G} } \longrightarrow
\dfrac{S}{J_{G_v}} \oplus \dfrac{S}{(x_v,y_v)+J_{G \setminus v}} \longrightarrow \dfrac{S}{(x_v,y_v)+J_{G_v \setminus v} } \longrightarrow 0.
\end{equation}
We first establish a connection between the number of internal vertices of $G$, $G_v$ and $G\setminus v$.
\begin{lemma}\label{ivlemma}
	Let $G$ be a  graph on $[n]$. If $v$ is an internal vertex of $G$, then $\iv(G)>\iv(G_v)$ and $\iv(G)>\iv(G \setminus v)$.
\end{lemma}
\begin{proof}
If $\iv(G_v)=0$, then $\iv(G)>\iv(G_v).$ Let $\iv(G_v)=k >0$. Let $v_1,\ldots, v_k$ be the internal
vertices of $G_v$. Note that $v$ is a free vertex of $G_v$. 
Assume that $v_{r+1}, \ldots , v_k \in N_G(v)=N_{G_v}(v)$.
For $r+1 \leq i \leq k$, $v_i$ is not a free vertex of 
$G_v$. Therefore, $v_i$ is  not a free vertex of $G$.
 For $1 \leq i \leq r$, $v_i$ 
is an internal vertex of $G$, since $v_i \notin N_G(v)$. As $v$ is an internal vertex
of $G$, we have $\iv(G) \geq k+1$. Let $u$ be a free vertex of $G$. In $G\setminus v$, $u$ is either
a free vertex or an isolated vertex, i.e., $u$ is not an internal vertex of $G \setminus v$. Hence, 
		$\iv(G) >\iv(G \setminus v)$.
\end{proof}

So far Conjecture \ref{con} has been proved  only for chordal graphs, (see \cite[Theorem 3.5]{MKM2018}). 
There have been  no attempts on non-chordal graphs. Now, we prove Conjecture \ref{con} for a class of non-chordal graphs. 

Let $H$ be a connected closed graph on $[n]$ such that $S/J_H$ is Cohen-Macaulay. By \cite[Theorem 3.1]{her1},
there exist integers $1=a_1 < a_2< \cdots < a_s <a_{s+1} = n$ such that $F_i=[a_i,a_{i+1}]$,
for $1\leq i \leq s$ and $F_1,\ldots,F_s$ is a leaf order of $\Delta(H)$. Set $e=\{1,n\}$. The graph $G=H \cup \{e\}$ 
is called the \textit{quasi-cycle} graph associated with $H$. In \cite{FM}, the Hilbert series of the binomial edge ideal of quasi-cycles was studied.
\begin{remark}\label{quasi-rmk}
	Let $G$ be the quasi-cycle graph associated with a Cohen-Macaulay closed graph $H$. 
	Let  $F_1,\ldots,F_s$ be a leaf order of $\Delta(H)$. If $H \neq P_3$, then $\iv(G)\geq s$ 
	and $\iv(H) =s-1$. If $s=2$, then $G$ is a chordal graph. If $s>2$, then $G$ has an induced 
	cycle  of length $s+1$ on the vertex set $\{a_1,a_2,\ldots,a_s,a_{s+1}\}$. Thus, for $s>2$, $G$ is not a  chordal graph. 
\end{remark}
\begin{example}
	Let $H$ and $G$	be the graphs shown in the figure below. It can be observed that $H$ is a 
	Cohen-Macaulay closed graph and $G$ is a quasi-cycle graph associated with $H$.
	
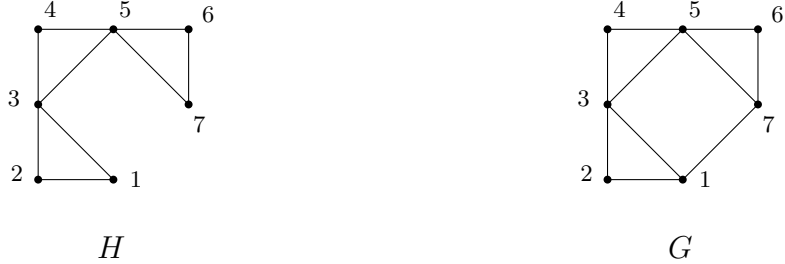
	\begin{minipage}{\linewidth}
		\begin{minipage}{.45\linewidth}
			\captionsetup[figure]{labelformat=empty}
			\begin{figure}[H]
				\begin{tikzpicture}[scale=1]		
				\draw (-1,4)-- (-1,3);
				\draw (-1,3)-- (0,4);
				\draw (-1,4)-- (0,4);
				\draw (0,4)-- (1,4);
				\draw (1,4)-- (1,3);
				\draw (0,4)-- (1,3);
				\draw (-1,3)-- (-1,2);
				\draw (-1,2)-- (0,2);
				\draw (-1,3)-- (0,2);
				\begin{scriptsize}
				\fill  (-1,4) circle (1.5pt);
				\draw (-0.84,4.26) node {$4$};
				\fill (0,4) circle (1.5pt);
				\draw (0.16,4.26) node {$5$};
				\fill (-1,3) circle (1.5pt);
				\draw (-1.32,3.08) node {$3$};
				\fill (1,3) circle (1.5pt);
				\draw (1.14,2.74) node {$7$};
				\fill (1,4) circle (1.5pt);
				\draw (1.26,4.2) node {$6$};
				\fill (-1,2) circle (1.5pt);
				\draw (-1.28,2.08) node {$2$};
				\fill (0,2) circle (1.5pt);
				\draw (0.3,2) node {$1$};
				\end{scriptsize}
				\end{tikzpicture}
				\caption{$H$}
			\end{figure}
		\end{minipage}
		\begin{minipage}{.45\linewidth}
			\captionsetup[figure]{labelformat=empty}
			\begin{figure}[H]
				\begin{tikzpicture}[scale=1]		
				\draw (-1,4)-- (-1,3);
				\draw (-1,3)-- (0,4);
				\draw (-1,4)-- (0,4);
				\draw (0,4)-- (1,4);
				\draw (1,4)-- (1,3);
				\draw (0,4)-- (1,3);
				\draw (-1,3)-- (-1,2);
				\draw (-1,2)-- (0,2);
				\draw (0,2)-- (1,3);
				\draw (-1,3)-- (0,2);
				\begin{scriptsize}
				\fill  (-1,4) circle (1.5pt);
				\draw (-0.84,4.26) node {$4$};
				\fill (0,4) circle (1.5pt);
				\draw (0.16,4.26) node {$5$};
				\fill (-1,3) circle (1.5pt);
				\draw (-1.32,3.08) node {$3$};
				\fill (1,3) circle (1.5pt);
				\draw (1.14,2.74) node {$7$};
				\fill (1,4) circle (1.5pt);
				\draw (1.26,4.2) node {$6$};
				\fill (-1,2) circle (1.5pt);
				\draw (-1.28,2.08) node {$2$};
				\fill (0,2) circle (1.5pt);
				\draw (0.3,2) node {$1$};
				\end{scriptsize}
				\end{tikzpicture}
				\caption{$G$}			
			\end{figure}
		
		\end{minipage}
	\end{minipage}
	
\end{example}

For a graph $G$, a maximal subgraph of $G $ without a cut vertex is called a  \textit{block} of $G$. A graph $G$ is said to be a \textit{block} graph if each  block of $G$ is a clique.
A block $B$ of a graph $G$ is called a \textit{quasi-block} if $B$ is a quasi-cycle other than $K_3$. 
\begin{definition}
	A graph $G$ is said to be a \textit{quasi-block} graph if $G$ satisfies the following:
	\begin{enumerate}
		\item Each block of $G$ is either a clique or a quasi-block.
		\item If $v$ is an internal vertex of a quasi-block    $B$, then for any $u \in N_G(v) \setminus V(B)$, $u$ is not an internal vertex of any block.
	\end{enumerate}
\end{definition}
One can note that a quasi-block graph need  not be a chordal graph. We denote by $\qc(G)$, the number of quasi-blocks in $G$.
\begin{remark}\label{quasi-rmk1}
	If $G$ is a connected quasi-block graph and $v$ is an internal vertex of $G$, then $G\setminus v$ is a quasi-block graph.
	If $B$ is a quasi-block of $G$ and $v \in V(B)$ is an internal vertex of $B$, then $G\setminus v$ is a quasi-block graph with $\qc(G\setminus v)=\qc(G)-1$. 
\end{remark} 
\begin{theorem}\label{quasi-block}
	Let $G$ be a quasi-block graph. Then, $\reg(S/J_G) \leq c(G)$.
\end{theorem} 
\begin{proof}
Let $G_1,\ldots,G_c$ be the connected components of $G$. Set $S_{G_i}=K[x_j,y_j :j\in V(G_i)]$.
Then $S/J_G\cong S_{G_1}/J_{G_1}\otimes \cdots \otimes S_{G_c}/J_{G_c}$ which implies that 
$\reg( S/J_G)=\reg( S_{G_1}/J_{G_1})+\cdots + \reg( S_{G_c}/J_{G_c})$. Also, $c(G)= c(G_1)+\cdots+c(G_c)$.
		Therefore, without loss of generality, we may assume that $G$ is a connected graph. We prove the result by induction on $\qc(G)$. 
	If $\qc(G)=0$, then every block of $G$ is a clique, thus $G$ is a block graph. By \cite[Theorem 3.9]{EZ}, the assertion follows. 
	
	Assume that $\qc(G)>0$. Let $B_1,\ldots,B_{\qc(G)}$ be the quasi-blocks of $G$. Now, we proceed by induction on the number of internal
	vertices of $B=B_{\qc(G)}$. It follows from Remark \ref{quasi-rmk} that $\iv(B) \geq 2$. If $\iv(B)=2$, then there
	exists an edge $e$ such that $B\setminus e$ is a  Cohen-Macaulay closed graph with exactly one internal vertex. Let $v \in V(B)$ be the 
	internal vertex of $B\setminus e$. Therefore, $v$ is also an internal vertex of $G$. By Lemma \ref{ohlemma}, $J_G=J_{G_v} \cap ((x_v,y_v)+J_{G\setminus v})$.
	Let $B_v$ be the block of $G_v$ which contains $v$. Note that $B_v$ is a clique as $\iv(B) =2$.
	Since $v$ is an internal vertex of $B$, $v \notin V(B_i)$, for $i=1,\ldots, \qc(G)-1$ and hence, $\qc(G_v) = \qc(G)-1$.   Therefore, by
	induction $\reg(S/J_{G_v})\leq c(G_v) < c(G)$. It follows from Remark \ref{quasi-rmk1} and  induction that
	$\reg(S/((x_v,y_v)+J_{G \setminus v})) \leq c(G\setminus v)  \leq c(G) $. Since $G_v \setminus v$ is an induced subgraph of
	$G_v$, by \cite[Proposition 8]{KM2}, $\reg(S/((x_v,y_v)+J_{G_v \setminus v}))\leq \reg(S/J_{G_v})< c(G)$. Thus, it follows 
	from Lemma \ref{regularity-lemma} and the short exact sequence (\ref{ses1}) that $\reg(S/J_G)\leq c(G)$.
	
	Now, assume that $\iv(B)>2$.  Let  $v \in V(B)$ be an internal vertex of $B \setminus e$. Therefore, $v$ is 
	an internal vertex of $G$. Again, by Lemma \ref{ohlemma}, $J_G =J_{G_v} \cap ((x_v,y_v)+J_{G \setminus v})$.
	By Remark \ref{quasi-rmk1} and induction, $\reg(S/((x_v,y_v)+J_{G\setminus v})) \leq c(G\setminus v) \leq c(G)$. 
	Let $B_v$ be the  block of $G_v$ which contains $v$. If $\iv(B)=3$, then $B_v$ is a clique. Therefore, $\qc(G_v) =\qc(G)-1$ and
	hence, by induction $\reg(S/J_{G_v})\leq c(G_v)< c(G)$. If $\iv(B)>3$, then $B_v$ is a quasi-block. By Lemma \ref{ivlemma}, $\iv(B_v) <\iv(B)$. Since $v$ is an internal vertex of $B$, $v \notin V(B_i)$, for $i=1,\ldots, \qc(G)-1$ and hence, $\qc(G_v) = \qc(G)$. Therefore, $G_v$ is a quasi-block graph with $\qc(G_v)=\qc(G)$ and  $\iv(B_v) <\iv(B)$.
	By induction, $\reg(S/J_{G_v})\leq c(G_v)< c(G)$, since $\iv(B_v) <\iv(B)$. It follows from \cite[Proposition 8]{KM2}
	that $\reg(S/((x_v,y_v)+J_{G_v \setminus v})) \leq \reg(S/J_{G_v}) < c(G)$ as $G_v \setminus v$ is an induced subgraph of
	$G_v$. Hence, using Lemma \ref{regularity-lemma} in  the short exact sequence (\ref{ses1}), we conclude that  $\reg(S/J_G) \leq c(G)$.
\end{proof}
The following example illustrates that the upper bound obtained in Theorem \ref{quasi-block} is tight.
\begin{example}
Let $G$ be a quasi-cycle graph such that for some $1 \leq i \leq s-1$,
$|F_i|>2$ and $|F_{i+1}| >2$. Then $c(G) = l(G\setminus a_{i+1}) \leq l(G)$. Hence, it follows from \cite[Theorem 1.1]{MM} and Theorem \ref{quasi-block} that $\reg(S/J_G)=c(G)$. 
\end{example}
Also, the upper bound obtained in Theorem \ref{quasi-block} can be strict upper bound. For example, 
$G=C_n$ is a quasi-block graph such that $c(G)=n$ and by \cite[Corollary 16]{SZ}, $\reg(S/J_G)=c(G)-2<c(G)$.

Let $H$ be a connected closed graph on $[m]$ such that $S_H/J_H$ is Cohen-Macaulay. 
By \cite[Theorem 3.1]{her1}, there exist integers $1=a_1 < a_2< \cdots < a_s <a_{s+1} = m$
such that $F_i=[a_i,a_{i+1}]$, for $1\leq i \leq s$ and $F_1,\ldots,F_s$ is a leaf order of $\Delta(H)$.
Set $F_{s+1}=[m,n] \cup \{1\}$. The graph $G$  on the vertex set $[n]$ and edge set 
$E(G) =E(H) \cup \{\{i,j\}: i \neq j, i,j \in F_{s+1}\}$ is called a \textit{semi-cycle} graph associated with $H$.

\begin{example}
	Let $H$ and $G$	be the graphs shown in the figure below. Then, it can be seen that $H$ is 
	a Cohen-Macaulay closed graph and $G$ is a semi-cycle graph associated with $H$.
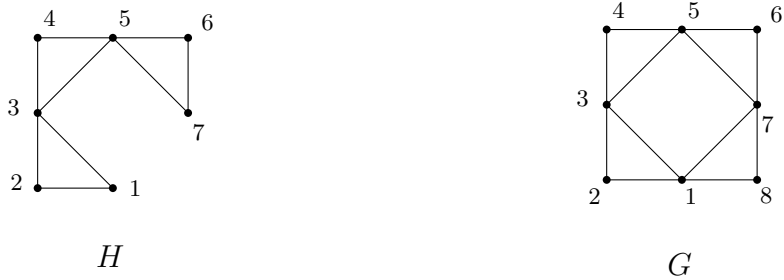
	\begin{minipage}{\linewidth}
		\begin{minipage}{.45\linewidth}
			\captionsetup[figure]{labelformat=empty}
			\begin{figure}[H]
				\begin{tikzpicture}[scale=1]		
				\draw (-1,4)-- (-1,3);
				\draw (-1,3)-- (0,4);
				\draw (-1,4)-- (0,4);
				\draw (0,4)-- (1,4);
				\draw (1,4)-- (1,3);
				\draw (0,4)-- (1,3);
				\draw (-1,3)-- (-1,2);
				\draw (-1,2)-- (0,2);
				\draw (-1,3)-- (0,2);
				\begin{scriptsize}
				\fill  (-1,4) circle (1.5pt);
				\draw (-0.84,4.26) node {$4$};
				\fill (0,4) circle (1.5pt);
				\draw (0.16,4.26) node {$5$};
				\fill (-1,3) circle (1.5pt);
				\draw (-1.32,3.08) node {$3$};
				\fill (1,3) circle (1.5pt);
				\draw (1.14,2.74) node {$7$};
				\fill (1,4) circle (1.5pt);
				\draw (1.26,4.2) node {$6$};
				\fill (-1,2) circle (1.5pt);
				\draw (-1.28,2.08) node {$2$};
				\fill (0,2) circle (1.5pt);
				\draw (0.3,2) node {$1$};
				\end{scriptsize}
				\end{tikzpicture}
				\caption{$H$}
			\end{figure}
		\end{minipage}
		\begin{minipage}{.45\linewidth}
			\captionsetup[figure]{labelformat=empty}
			\begin{figure}[H]
				\begin{tikzpicture}[scale=1]
				\draw (-1,4)-- (-1,3);
				\draw (-1,3)-- (0,4);
				\draw (-1,4)-- (0,4);
				\draw (0,4)-- (1,4);
				\draw (1,4)-- (1,3);
				\draw (0,4)-- (1,3);
				\draw (-1,3)-- (-1,2);
				\draw (-1,2)-- (0,2);
				\draw (0,2)-- (1,3);
				\draw (-1,3)-- (0,2);
				\draw (1,2)-- (0,2);
				\draw (1,3)-- (1,2);
				\begin{scriptsize}
				\fill  (-1,4) circle (1.5pt);
				\draw (-0.84,4.26) node {$4$};
				\fill (0,4) circle (1.5pt);
				\draw (0.16,4.26) node {$5$};
				\fill (-1,3) circle (1.5pt);
				\draw (-1.32,3.08) node {$3$};
				\fill (1,3) circle (1.5pt);
				\draw (1.14,2.74) node {$7$};
				\fill (1,4) circle (1.5pt);
				\draw (1.26,4.2) node {$6$};
				\fill (-1,2) circle (1.5pt);
				\draw (-1.16,1.78) node {$2$};
				\fill (0,2) circle (1.5pt);
				\draw (0.12,1.78) node {$1$};
				\fill (1,2) circle (1.5pt);
				\draw (1.12,1.8) node {$8$};
				\end{scriptsize}
				\end{tikzpicture}
				\caption{$G$}
			\end{figure}
		\end{minipage}
	\end{minipage}
\end{example}
\begin{remark}
	Let $G$ be a semi-cycle graph associated with the Cohen-Macaulay closed graph $H$. 
	Let  $F_1,\ldots,F_s$ be a leaf order of $\Delta(H)$. If $H \neq P_3$, then $\iv(G)\geq s$ and $\iv(H) =s-1$.
	If $s=2$, then $G$ is a chordal graph. If $s>2$, then $G$ has an induced cycle on the vertex set $\{a_1,a_2,\ldots,a_s,a_{s+1}\}$
	of length $s+1$ and hence, $G$ is not chordal. Note that every quasi-cycle is a semi-cycle, but a semi-cycle need not be a quasi-cycle.
	Also, $C_n$ is a semi-cycle graph.
\end{remark} 

A block $B$ of a graph $G$ is said to be a \textit{semi-block} if $B$ is a semi-cycle with $B \neq K_3$.
A graph $G$ is said to be a \textit{semi-block} graph if all except one block are cliques and the block which 
is not a clique is a semi-block.

We now prove Conjecture \ref{con} for semi-block graphs.

\begin{theorem}\label{semi-reg}
	Let $G$ be a semi-block graph. Then, $\reg(S/J_G) \leq c(G)$.
\end{theorem}
\begin{proof}
	Let $B$ be the semi-block of $G$. If $B$ is a quasi-block of $G$, then $G$ is a quasi-block graph. Therefore, by Theorem \ref{quasi-block}, the assertion follows.	 Assume that $B$ is not a quasi-block, i.e., $\iv(B) \geq 3$ 
	and for each $1 \leq i \leq s+1$, $|F_i| \geq 3$.
	We proceed by induction on $\iv(B)$. For $\iv(B) =3$, we claim that $\reg(S/J_G)\leq c(G)-1$.
	Let $v \in V(B)$ be an internal vertex of $B$. Therefore, $v$ is also an internal vertex of $G$. It follows
	from Lemma \ref{ohlemma} that $J_G =J_{G_v} \cap ((x_v,y_v)+J_{G \setminus v})$. Let $B_v$ be the block of $G_v$ which contains $v$.
	Note that $B_v$ is a chordal graph with $c(B_v)=2$ and hence, $G_v$ is a chordal graph with $c(G_v)\leq c(G) -2$. It follows from \cite[Theorem 3.5]{MKM2018} that $\reg(S/J_{G_v}) \leq c(G_v)\leq c(G)-2$. As $G\setminus v$ 
	is a block graph with $c(G \setminus v) \leq c(G)-1$, by \cite[Theorem 3.9]{EZ}, $\reg(S/((x_v,y_v)+J_{G\setminus v})) \leq c(G\setminus v)\leq c(G)-1$. Since $G_v\setminus v$ is an induced subgraph of 
	$G_v$, by \cite[Proposition 8]{KM2}, $\reg(S/((x_v,y_v)+J_{G_v \setminus v})) \leq \reg(S/J_{G_v})$. Thus, it follows from Lemma \ref{regularity-lemma} and 
	the short exact sequence (\ref{ses1}) that $\reg(S/J_G)\leq c(G)-1$.
	
	Assume that $\iv(B)>3$ and let $v$ be an internal vertex of $B$. Then, $G \setminus v$ is a block graph with $c(G\setminus v)\leq c(G)$.
	Therefore, by \cite[Theorem 3.9]{EZ}, $\reg(S/((x_v,y_v)+J_{G\setminus v}))\leq c(G)$. Note that $B_v$ is a
	semi-block. It follows from Lemma \ref{ivlemma} that $\iv(B_v)<\iv(B)$. If $\iv(B) =4$, then $c(G_v)\leq c(G)$ and $\iv(B_v)=3$. Therefore, 
	$\reg(S/J_{G_v})\leq c(G_v)-1 <c(G)$. If $\iv(B)>4$, then $c(G_v)<c(G)$ and hence, by induction, $\reg(S/J_{G_v}) \leq c(G_v)< c(G)$.
	It follows from \cite[Proposition 8]{KM2} that $\reg(S/((x_v,y_v)+J_{G_v \setminus v})) \leq \reg(S/J_{G_v})<c(G)$. 
	Hence, by Lemma \ref{regularity-lemma} and the short exact sequence(\ref{ses1}), $\reg(S/J_G)\leq c(G)$.
\end{proof}

For a chordal graph $G$, we recall a result from \cite{MKM2018} which connects  $c(G)$ and $c(G_v)$.
\begin{lemma}\cite[Lemma 3.4]{MKM2018}\label{l}
	Let $G$ be a chordal graph and $v$ be a  vertex of $G$ which lies in $t$ maximal cliques of $G$. Then,  $c(G_v)\leq c(G)-t+1$. In particular, if $t \geq 2$, then $c(G_v)<c(G)$.
\end{lemma}
As a consequence of Lemma \ref{l}, we have the following.
\begin{corollary}\label{chordal-graph}
Let $G$ be a chordal graph. If $v$ is an internal vertex of $G$, then $G_v$ is a chordal graph and $c(G_v)<c(G)$.
\end{corollary}
\begin{proof}
The assertion that $G_v$ is a chordal graph follows from the  second paragraph of the proof of \cite[Theorem 3.5]{MKM2018}. Since $v$ is an internal vertex, $v$ belongs to at least two maximal cliques and hence, by Lemma \ref{l}, $c(G_v)<c(G)$.
\end{proof}
\begin{remark}\label{chordal-graph1}
	If $G$  is a chordal graph, then $G\setminus v$ is a chordal graph for any $v$. If $v$ is an internal vertex of $G$, then $c(G \setminus v) \leq c(G)$.
\end{remark}
 In \cite{MKM2018}, Rouzbahani Malayeri et al. proved Saeedi Madani-Kiani conjecture for chordal graphs. They proved the result by induction on $n+c(G)$. 
We give another proof of the same by induction on $\iv(G)$.
\begin{theorem}\cite[Theorem 3.5]{MKM2018}\label{chordal-reg}
	Let $G$ be a chordal graph. Then, $\reg(S/J_G) \leq c(G)$.
\end{theorem}
\begin{proof}
Without loss of generality,  we may assume that $G$ is connected.	
We prove the assertion by induction
on $\iv(G).$ If $\iv(G)=0,$ then $G$ is a complete graph. Therefore, $c(G) =1$ and the result follows from Eagon-Northcott complex \cite{EN}.
Assume that $\iv(G)>0$ and if $H$ is a chordal graph with $\iv(H) < \iv(G)$, then $\reg(S_H/J_H) \leq c(H)$,
where $S_H=K[x_i,y_i:i \in V(H)]$. Let  $v$ be an internal vertex of $G$. 
It follows from Lemma \ref{ohlemma} that $J_G =J_{G_v} \cap ((x_v,y_v)+J_{G \setminus v}).$

By Corollary \ref{chordal-graph}, $G_v$  is a connected chordal graph. Now, by Lemma \ref{ivlemma}, $\iv(G_v)<\iv(G)$ and hence, by induction, $\reg(S/J_{G_v})\leq c(G_v).$

It follows from Lemma \ref{ivlemma} that $\iv(G \setminus v)<\iv(G)$. If $G \setminus v$ is connected,
then by induction, $\reg(S/((x_v,y_v)+J_{G\setminus v})) \leq c(G\setminus v)$. If $v$ is a cut vertex, 
then let $H_1,\ldots,H_p$ be the connected components of  $G\setminus v$. By induction, $\reg(S_{H_i}/J_{H_i}) \leq c(H_i)$
for each $i$. Therefore, $$\reg(S/((x_v,y_v)+J_{G\setminus v})) = \underset{i \in[p]}{\sum} \reg(S_{H_i}/J_{H_i}) \leq \underset{i \in[p]}{\sum}c(H_i)=c(G \setminus v).$$
	
The graph $G_v \setminus v$ is  an induced subgraph of $G_v$. Therefore, by \cite[Proposition 8]{KM2},
$\reg(S/((x_v,y_v)+J_{G_v \setminus v}))\leq \reg(S/J_{G_v})\leq c(G_v)$.
	
Now, by Corollary \ref{chordal-graph}, $c(G_v)<c(G)$ and by Remark \ref{chordal-graph1}, $c(G\setminus v)\leq c(G)$.
Hence, by applying Lemma \ref{regularity-lemma} in the short exact sequence (\ref{ses1}), we get the desired result.
\end{proof}

 We  recall notation of decomposability from \cite{Rauf}. 
A graph $G$ is called \textit{decomposable}, if there exist subgraphs $G_1$ and $ G_2$ such that $G$ is obtained by identifying 
a free vertex $v_1$ of $G_1$ with a free vertex $v_2$ of $G_2 $, i.e., 
$G= G_1 \cup G_2$ with $V(G_1)\cap V(G_2)=\{v\}$ such that $v$ is a 
free vertex of both $G_1$ and $G_2$.

A graph $G$ is called \textit{indecomposable}, if it is not decomposable. Up to ordering, $G$
has  a unique decomposition into indecomposable subgraphs, i.e., there exist 
$G_1,\ldots,G_r$ indecomposable induced subgraphs of $G$ with 
$G=G_1\cup \cdots \cup G_r$ such that for each $i \neq j$, either 
$V(G_i) \cap V(G_j) = \emptyset$ or $V(G_i) \cap V(G_j) =\{v\}$ and $v$ is a  free vertex 
of both $G_i$ and $G_j$.

It follows from \cite[Theorem 3.1]{JNR} that if $G = G_1 \cup \cdots \cup G_r$ is a decomposition into indecomposable graphs, then
$\reg(S/J_G) = \sum_{i=1}^r \reg(S/J_{G_i})$. Therefore,  we consider indecomposable graphs to study the regularity.

Let $u,v \in V(G)$ be such that $e=\{u,v\} \notin E(G)$, 
then we denote by $G_e$, the graph on the  vertex set $V(G)$ and edge set 
$E(G_e) = E(G) \cup \{\{x,y\} : x,\; y \in N_G(u) \; or \; x,\; y \in N_G(v) \}$. An edge $e$ is said to be a 
\textit{cut edge} if the number of connected components of $G\setminus e$ is  more than  the number of connected components of 
$G$. 

In the following theorem, we give a sufficient condition for a chordal graph whose regularity is not  maximal.

\begin{theorem}\label{chordal-cutedge}
Let $G$ be a connected indecomposable chordal graph on $[n]$. If $G$ has a cut edge, then $\reg(S/J_G) <c(G)$.
\end{theorem}
\begin{proof}
Let $e=\{u,v\}$ be a cut edge of $G$. Let $H_1$ and $H_2$ be connected components of $G\setminus e$ with $u \in V(H_1)$ and $v \in V(H_2)$. 
Since $G$ is an  indecomposable graph, $u$ and $v$ are internal vertices of $H_1$ and $H_2$, respectively. 
Note that $(G\setminus e)_e =(H_1)_u\sqcup (H_2)_v$. By Corollary \ref{chordal-graph}, $(H_1)_u$ and $(H_2)_v$ are  chordal graphs and $c((G\setminus e)_e) =c((H_1)_u)+c((H_2)_v)\leq c(H_1)+c(H_2)-2= c(G)-3$.
Therefore, by Theorem \ref{chordal-reg}, $\reg(S/J_{(G\setminus e)_e}) =\reg(S_{H_1}/J_{(H_1)_u}) +\reg(S_{H_2}/J_{(H_2)_v})\leq c(G)-3$.
Also, $G\setminus e=H_1 \sqcup H_2$ is a chordal graph and
$c(G\setminus e) =c(G)-1$. Therefore, by Theorem \ref{chordal-reg}, 
$\reg(S/J_{G\setminus e}) =\reg(S_{H_1}/J_{H_1})+\reg(S_{H_2}/J_{H_2})\leq c(G)-1$. It follows from  Lemma \cite[Proposition 2.1(a)]{KM3}   that $\reg(S/J_G) < c(G)$.
\end{proof}
The following example illustrates that the assumption in Theorem
\ref{chordal-cutedge}, that $G$ contains a cut edge, is not a
necessary condition.  First, we recall the definition of join of graphs.  Let $H$ and $H'$ be two graphs  with the vertex sets $[p]$ and $[q]$, respectively. The \textit{join} of $H$ and $H'$, denoted by $H*H'$, 
is the graph with vertex set $[p] \sqcup [q]$ and the edge set $E(H*H')= E(H) \cup E(H') \cup \{\{i,j\}| i \in [p], j \in [q]\}$.
\begin{example}
For $n \geq 3$, let  $G=P_2 * K_n^c$, where $K_n^c$ is a graph on $n$ vertices and has no edges. Note that $G$ is an indecomposable chordal graph on $n+2$ vertices. Observe that $c(G) =n$ and $G$ has no cut edge. It follows from  \cite[Theorem 2.1]{MK2018} that $\reg(S/J_G) =\max \{\reg(S/J_{P_2}), 2\} =2$.
\end{example}

\section{Regularity of binomial edge ideals of Jahangir graphs}
 In this section, we obtain the regularity of binomial edge ideals of Jahangir graphs.   First, we recall the definition of Jahangir graph. 
\begin{definition}
The Jahangir graph denoted by $J_{m,n}$ is a graph  on the vertex set $[mn+1]$, for $m \geq 1$ and $n \geq 3$, such that the induced
subgraph on $[mn]$ is  $C_{mn}$ and the neighborhood of vertex $mn+1$ is $\{1,m+1,\ldots,m(n-1)+1\}$.
\end{definition}

\begin{example} We give illustrations of  $J_{1,8}$ and $J_{2,4}$ below.
	
	\begin{minipage}{\linewidth}
		\begin{minipage}{.45\linewidth}
			\captionsetup[figure]{labelformat=empty}
			\begin{figure}[H]
				\begin{tikzpicture}[scale=1]
				\draw (0.08,3.52)-- (1,4);
				\draw (2,3.58)-- (1,4);
				\draw (2,3.58)-- (2.54,2.5);
				\draw (2.54,2.5)-- (2.04,1.48);
				\draw (2.04,1.48)-- (1,1);
				\draw (1,1)-- (0.04,1.46);
				\draw (0.04,1.46)-- (-0.38,2.46);
				\draw (-0.38,2.46)-- (0.08,3.52);
				\draw (1.04,2.44)-- (0.08,3.52);
				\draw (1.04,2.44)-- (0.04,1.46);
				\draw (1.04,2.44)-- (2,3.58);
				\draw (1.04,2.44)-- (2.04,1.48);
				\draw (1.04,2.44)-- (1,1);
				\draw (-0.38,2.46)-- (1.04,2.44);
				\draw (1.04,2.44)-- (1,4);
				\draw (1.04,2.44)-- (2.54,2.5);
				\begin{scriptsize}
				\fill (1,4) circle (1.5pt);
				\draw (1.14,4.26) node {$1$};
				\fill (0.08,3.52) circle (1.5pt);
				\draw (-0.1,3.86) node {$2$};
				\fill (2,3.58) circle (1.5pt);
				\draw (2.16,3.84) node {$8$};
				\fill (0.04,1.46) circle (1.5pt);
				\draw (-0.14,1.3) node {$4$};
				\fill (2.04,1.48) circle (1.5pt);
				\draw (2.22,1.38) node {$6$};
				\fill (1,1) circle (1.5pt);
				\draw (1.2,0.88) node {$5$};
				\fill (2.54,2.5) circle (1.5pt);
				\draw (2.7,2.76) node {$7$};
				\fill (-0.38,2.46) circle (1.5pt);
				\draw (-0.58,2.58) node {$3$};
				\fill (1.04,2.44) circle (1.5pt);
				\draw (1.16,2.1) node {$9$};
				\end{scriptsize}
				\end{tikzpicture}
				\caption{$J_{1,8}$}
			\end{figure}
		\end{minipage}
		\begin{minipage}{.45\linewidth}	
			\captionsetup[figure]{labelformat=empty}
			\begin{figure}[H]
				\begin{tikzpicture}[scale=1]
				\draw (0.56,2.98)-- (1,4);
				\draw (1.44,2.98)-- (1,4);
				\draw (1.44,2.98)-- (2.52,2.48);
				\draw (2.52,2.48)-- (1.46,2);
				\draw (1.46,2)-- (1,1);
				\draw (1,1)-- (0.58,2);
				\draw (0.58,2)-- (-0.38,2.46);
				\draw (-0.38,2.46)-- (0.56,2.98);
				\draw (1.04,2.44)-- (0.56,2.98);
				\draw (1.04,2.44)-- (0.58,2);
				\draw (1.04,2.44)-- (1.44,2.98);
				\draw (1.04,2.44)-- (1.46,2);
				\begin{scriptsize}
				\fill (1,4) circle (1.5pt);
				\draw (1.24,4) node {$2$};
				\fill (0.56,2.98) circle (1.5pt);
				\draw (0.38,3.32) node {$3$};
				\fill (1.44,2.98) circle (1.5pt);
				\draw (1.6,3.24) node {$1$};
				\fill (0.58,2) circle (1.5pt);
				\draw (0.4,1.84) node {$5$};
				\fill (1.46,2) circle (1.5pt);
				\draw (1.64,1.9) node {$7$};
				\fill (1,1) circle (1.5pt);
				\draw (1.28,1.1) node {$6$};
				\fill (2.52,2.48) circle (1.5pt);
				\draw (2.68,2.74) node {$8$};
				\fill (-0.38,2.46) circle (1.5pt);
				\draw (-0.5,2.84) node {$4$};
				\fill (1.04,2.44) circle (1.5pt);
				\draw (1.06,2.18) node {$9$};
				\end{scriptsize}
				\end{tikzpicture}
				\caption{$J_{2,4}$}
			\end{figure}
		\end{minipage}
	\end{minipage}
\end{example}
 One can note that if $m >1$, then $J_{m,n}$ is not a chordal graph and $c(J_{m,n})=n(m+1)$. 
 It follows from \cite[Theorem 3.2]{KM3} that $$\reg(S/J_{J_{m,n}}) \leq mn-1<c(J_{m,n}).$$ 
 Since $l(J_{m,n})=mn-2$, by \cite[Theorem 1.1]{MM}, $mn-2 \leq \reg(S/J_{J_{m,n}})$.  If $m=1$, then $J_{1,n}$ is the wheel graph $W_n$ and hence, by virtue of \cite[Theorem 2.1]{MK2018}, $\reg(S/J_{J_{1,n}})=n-2$. 
 We now prove the same for $m\geq 2$.	
 First, we recall a result due to Ene et al. from \cite{ERT19} which is useful to compute the regularity of  binomial edge ideals of Jahangir graphs.
 
\begin{theorem}\cite[Theorem 2.1]{ERT19}\label{new-up}
Let $G$ be a connected graph on $[n]$. Then, $\reg(S/J_G) \leq n-\omega(G)
+1$.
\end{theorem}

We now obtain the  regularity of binomial edge ideals of Jahangir graphs.  To compute the regularity of $J_{m,n}$, we need the following lemma. The idea of this lemma is based on \cite[Lemma 3.3]{ERT19}.
\begin{lemma}\label{new-upp}
	Let $G$ be a connected graph on $[n]$. If $G$ has an internal vertex $v$ such that $\deg_G(v) \geq 4$ and $G\setminus v$ is not a path graph on $n-1$ vertices,  then $\reg(S/J_G) \leq n-3$.
\end{lemma}
\begin{proof}
	Since $v$ is an internal vertex, by Lemma \ref{ohlemma}, $J_G = J_{G_v} \cap ((x_v,y_v)+J_{G \setminus v})$. 
	Note that  $\omega(G_v) \geq 5$ and $\omega(G_v \setminus v) \geq 4$ as $\deg_G(v) \geq 4$. By Theorem \ref{new-up}, $\reg(S/J_{G_v}) \leq n-\omega(G_v)+1 \leq n-4$ and $\reg(S/((x_v,y_v)+J_{G_v \setminus v})) \leq (n-1)-\omega(G_v \setminus v)+1 \leq n-4$.  Since $G\setminus v$ is not a path graph, it follows from \cite[Theorem 3.2]{KM3} that $\reg(S/((x_v,y_v)+J_{G\setminus v})) \leq (n-1)-2=n-3$. Hence, the assertion follows from  Lemma \ref{regularity-lemma} and the short exact sequence \eqref{ses1}.
\end{proof}

We conclude this section by computing the regularity of binomial edge ideals of  Jahangir graphs.
\begin{theorem}
	 For $m \geq 2 \text{ and } n \geq 3$, let $G=J_{m,n}$. Then, $\reg(S/J_G) =mn-2$.
\end{theorem}
\begin{proof}
Since $l(G)=mn-2$, by \cite[Theorem 1.1]{MM}, $mn-2 \leq \reg(S/J_{G})$.  Therefore, it is enough to prove that $\reg(S/J_G) \leq mn-2$. Set $v=mn+1$. Observe that $v$ is an internal vertex of $G$. If $n \geq 4$, then $\deg_G(v)=n \geq 4$ and $G\setminus v$ is not a path graph on $mn$ vertices. Thus, by Lemma \ref{new-upp}, $\reg(S/J_G) \leq mn-2$.

Now, we assume that $n=3$. By Lemma \ref{ohlemma}, $J_G = J_{G_v} \cap ((x_v,y_v)+J_{G \setminus v})$. 
Note that $\omega(G_v) = 4$. Hence, by virtue of Theorem
\ref{new-up}, $\reg(S/J_{G_v})\leq (3m+1)-4+1 =3m-3$. Set $H =G_v \setminus v$. Then, $H$ is a connected graph on $3m$ vertices, $\deg_H(1) =4$ and $H \setminus 1$ is not a path graph on $3m-1$ vertices. Thus, by Lemma \ref{new-upp}, $\reg(S/((x_v,y_v)+J_{H})) \leq 3m-3$. It is clear that
$G \setminus v =C_{3m}$, thus  by \cite[Corollary 16]{SZ}, $\reg(S/((x_v,y_v)+J_{G \setminus v})) =3m-2$. By virtue of Lemma \ref{regularity-lemma} and the short exact sequence \eqref{ses1}, we get $\reg(S/J_G) \leq 3m-2$.
\end{proof}

\section{Hibi-Matsuda Conjecture}
Let $M =\underset{k \in \mathbb{N}} \bigoplus M_k$ be a finite graded $S$-module of Krull dimension $d$.
The  function $H_M : \mathbb{N}  \rightarrow \mathbb{N}$ defined as $H_M(k) = l(M_k)$ is called the Hilbert function of $M$.
The Hilbert series of $M$ is the generating series of the Hilbert function $H_M$ and is denoted by
$\Hilb_{M}(t)= \underset{k \in \mathbb{N}} \sum l(M_k) t^k$. By  \cite[Corollary 4.1.8]{bh}, there
exists a unique polynomial $h_M(t) \in \mathbb{Z}[t]$ such that  $h_M(1)\neq 0$ and $\Hilb_M(t) =h_M(t)/(1-t)^d$. The polynomial $h_M(t)$ is called the $h$-polynomial of $M$.

In this section, we give a sufficient condition for  Conjecture  \ref{Hibi-con} to
be  true. We prove that Conjecture \ref{Hibi-con} holds if
$S/J_G$ admits a unique extremal Betti number. In \cite{PZ}, Schenzel
and Zafar proved that complete bipartite graphs have unique extremal
Betti number. Zafar and Zahid proved that the $n$-cycle, $C_n$, has a  unique extremal Betti number, see \cite{SZ}.  In \cite{her2}, Herzog and Rinaldo characterized block graphs which admit unique 
extremal Betti number. In \cite{AR2}, we have  characterized generalized block graphs which admit unique extremal Betti number. 
\begin{theorem}\label{unique-extremal}
	Let $G$ be a connected graph on $[n]$. If $S/J_G$ admits a unique extremal Betti number, then $\reg(S/J_G) \leq \deg h_{S/J_G}(t)$.
\end{theorem}		
\begin{proof}
Set $p=\pd(S/J_G)$ and $r=\reg(S/J_G)$. If $G$ is not a complete graph, then by \cite[Theorem B]{AN2017},
$p \geq n+\kappa(G)-2$, where $\kappa(G)$ is the vertex connectivity of $G$. If $G$ is a complete graph, then $p =n-1$.
Note that for any connected graph $G$ if $G \neq K_n$, then $\kappa(G) \geq 1$. Therefore, for any connected graph $G$,  $ p \geq n-1$. Also, 
it follows from \cite[Corollary 3.3]{HH1} that  $\dim(S/J_G) \geq n+1$. Now, by \cite[Corollary 4.1.14]{bh}, 
$$\underset{i,j}\sum(-1)^i \beta_{i,j}^S(S/J_G)t^j =h_{S/J_G}(t)(1-t)^{2n-d},$$ where $d=\dim(S/J_G)$.    
Since $S/J_G$ has the unique extremal Betti number $\beta_{p,p+r}^S(S/J_G)$, we get $p+r =2 n-d + \deg h_{S/J_G}(t)$. Hence, the assertion follows.
\end{proof}
It is natural to ask if Conjecture \ref{Hibi-con} is true when $S/J_G$ has more than one extremal Betti numbers.
Here is one instance when   $S/J_G$ has two extremal Betti number and Conjecture \ref{Hibi-con} is true. 
A \textit{flower} graph $F_{h,k}(v)$ is a connected graph obtained  by identifying a free vertex as $v$, 
each of $h$ copies of the complete graph $K_3$ and $k$ copies of the star graph $K_{1,3}$ with $h +k \geq 3$. The flower graph was introduced
by Mascia and Rinaldo in \cite{CarlaR2018}. It follows from \cite[Theorem 3.4]{CarlaR2018} that  $F_{h,k}(v)$ has two extremal Betti numbers.

\begin{theorem}
	Let $G=F_{h,k}(v)$ be a flower graph. Then, $\reg(S/J_G) \leq \deg h_{S/J_G}(t)$.
\end{theorem}
\begin{proof}
It follows from \cite[Corollary 4.1.14]{bh} that $$\underset{i,j}\sum(-1)^i \beta_{i,j}(S/J_G)t^j =h_{S/J_G}(t)(1-t)^{2n-d},$$
where $d =\dim(S/J_G)$.
By \cite[Theorem 3.4]{CarlaR2018}, the degree of the polynomial on the left hand side of the above equation is $n+\iv(G)$.
Therefore, by comparing the degree, we get $n+\iv(G) =2n-d + \deg h_{S/J_G}(t)$. It follows from \cite[Corollary 3.5]{CarlaR2018} 
that $\reg(S/J_G)=\iv(G)+\cdeg(v)-1 =n -d+\deg h_{S/J_G}(t)+\cdeg(v)-1$. Let $T=\{v\}$. Note that, $c_G(T) = \cdeg(v)$. 
By \cite[Corollary 3.3]{HH1},  $d \geq n+\cdeg(v)-1$. Thus, we get $\reg(S/J_G) \leq \deg h_{S/J_G}(t)$.
\end{proof}	
Now, we provide a counterexample to Hibi-Matsuda conjecture. In
\cite{KK19}, Kahle and Kr\"usemann gave a counterexample to
Hibi-Matsuda conjecture. However, one can observe that the counterexample given in \cite{KK19} is not a chordal graph. Here, we provide a counterexample which is a block graph and hence, a chordal graph. 
\begin{example}The following graph is a counterexample to Hibi-Matsuda Conjecture.
		
\captionsetup[figure]{labelformat=empty}		
\begin{figure}[H]
		\begin{tikzpicture}[scale=1]
		\draw  (-3,2)-- (-1,2);
		\draw  (-1,2)-- (-2,1);
		\draw  (-2,1)-- (-3,2);
		\draw  (-3,2)-- (-3.01,3.44);
		\draw  (-3.01,3.44)-- (-4.63,2);
		\draw  (-4.63,2)-- (-3,2);
		\draw  (-3,2)-- (-3.03,0.6);
		\draw  (-3,2)-- (-4.59,0.58);
		\draw  (-4.59,0.58)-- (-3.03,0.6);
		\draw  (-2,1)-- (-2.59,0);
		\draw  (-2,1)-- (-1.43,0);
		\draw  (-1,2)-- (-1.01,3.4);
		\draw  (-1.01,3.4)-- (0.67,2);
		\draw  (0.67,2)-- (-1,2);
		\draw (-1,2)-- (-1.01,0.56);
		\draw  (-1.01,0.56)-- (0.61,0.58);
		\draw  (0.61,0.58)-- (-1,2);
		\begin{scriptsize}
		\draw  (-3,2) circle (1.5pt);
		\draw (-2.83,2.43) node {$1$};
		\draw  (-1,2) circle (1.5pt);
		\draw (-1.29,2.41) node {$2$};
		\draw (-2,1) circle (1.5pt);
		\draw(-2.39,1.07) node {$3$};
		\draw  (-3.01,3.44) circle (1.5pt);
		\draw (-2.75,3.49) node {$4$};
		\draw (-4.63,2) circle (1.5pt);
		\draw (-5,2.07) node {$5$};
		\draw  (-3.03,0.6) circle (1.5pt);
		\draw (-3.07,0.20) node {$6$};
		\draw  (-4.59,0.58) circle (1.5pt);
		\draw (-4.69,0.20) node {$7$};
		\draw  (-2.59,0) circle (1.5pt);
		\draw (-2.23,0.05) node {$8$};
		\draw(-1.43,0) circle (1.5pt);
		\draw (-1.17,0.05) node {$9$};
		\draw (-1.01,3.4) circle (1.5pt);
		\draw(-1.35,3.49) node {$10$};
		\draw (0.67,2) circle (1.5pt);
		\draw (0.95,2.19) node {$11$};
		\draw (-1.01,0.56) circle (1.5pt);
		\draw (-0.71,0.97) node {$12$};
		\draw  (0.61,0.58) circle (1.5pt);
		\draw (1,0.7) node {$13$};
		\end{scriptsize}
		\end{tikzpicture}
		\caption{$G$}
		\end{figure}
	\end{example}
	It follows from \cite[Theorem 4.2]{CarlaR2018} that $\reg(S/J_G) =6$.
We computed the Hilbert series of $S/J_G$ using Macaulay 2 package
\cite{M2}: $$\Hilb_{S/J_G}(t) =\dfrac{1+10t+38t^2+60t^3+19t^4-24t^5}{(1-t)^{18}}.$$
The polynomial $h_{S/J_G}(t)=1+10t+38t^2+60t^3+19t^4-24t^5$ is the  $h$-polynomial of $S/J_G$ and $\deg h_{S/J_G}(t) =5<\reg(S/J_G)$.  
Let $G_1,\ldots,G_k$ be $k$ copies of the graph $G$.  The graph $G^k$ is obtained  by identifying a free vertex of $G_i$ with a free vertex of $G_{i+1}$, 
i.e. $G^k=G_1 \cup \cdots \cup G_k$, $V(G_i) \cap V(G_{j})=\emptyset$, if $j \notin \{i-1,i+1\}$ and  $V(G^i) \cap V(G^{i+1})=\{u_i\}$, where $u_i$ is  a
free vertex of $G^i$ and $G^{i+1}$. Then, it follows from \cite[Theorem 3.1]{JNR} that $\reg(S_k/J_{G^k})=6k$, 
where $S_k =K[x_i,y_i: i \in V(G^k)]$. Also, by \cite[Corollary 3.3]{AR1}, 
$\deg h_{S_k/J_{G^k}}(t) =5k$. This shows that for any positive integer $k$, there is a graph $G^k$ such that $\reg(S_k/J_{G^k})-\deg h_{S_k/J_{G^k}}(t) =k$.

	In \cite{AH18}, de Alba and Hoang asked whether the initial ideal of the binomial edge ideal of a closed graph admits a unique extremal Betti number, (see \cite[Question 1]{AH18}). Recently, in \cite{SMD19}, Saeedi Madani and Kiani gave a negative answer to the above question, (see \cite[Theorem 4.6]{SMD19}). 
	We end this article by asking the following question.
	\begin{question}
		When does the binomial edge ideal of a graph admit a unique extremal Betti number?
	\end{question}
	
\vskip 2mm
\noindent
\textbf{Acknowledgements:} The author is grateful to his advisor A. V. Jayanthan for
his constant support, valuable ideas and suggestions. The  author thanks the National Board
for Higher Mathematics, India for the financial support. The author also wishes to express his sincere gratitude to
the anonymous referees whose comments helped improve the exposition
in great detail.
\bibliographystyle{plain}  %% or 
\bibliography{biblo}
\end{document}